\documentclass[12pt,a4paper]{article}
\usepackage{graphicx}
\usepackage{amsfonts}
\usepackage{amsmath}
\usepackage{amssymb}
\usepackage{amsthm}
\usepackage[ansinew]{inputenc}
\usepackage{latexsym}
\usepackage{tabularx}
\usepackage{xcolor}
\usepackage[active]{srcltx}
\allowdisplaybreaks
\newtheorem{theorem}{Theorem}

\newtheorem{lem}{Lemma}
\newtheorem{prop}{Proposition}

\newtheorem{rem}{Remark}

\newcommand{\norm}[1]{\left\Vert#1\right\Vert}
\newcommand{\abs}[1]{\left\vert#1\right\vert}

\newcommand{\To}{\rightarrow}

\newcommand{\bsa}{\boldsymbol{a}}

\newcommand{\bsk}{\boldsymbol{k}}
\newcommand{\bsl}{\boldsymbol{l}}

\newcommand{\bsx}{\boldsymbol{x}}

\newcommand{\bsh}{\boldsymbol{h}}

\newcommand{\bst}{\boldsymbol{t}}

\newcommand{\1}{\boldsymbol{1}}

\newcommand{\bsb}{\boldsymbol{b}}

\newcommand{\bsy}{\boldsymbol{y}}

\newcommand{\cH}{{\cal H}}

\newcommand{\bszero}{\boldsymbol{0}}
\newcommand{\bsone}{\boldsymbol{1}}
\newcommand{\rd}{\,\mathrm{d}}
\newcommand{\e}{\varepsilon}

\newcommand{\NN}{\mathbb{N}}

\newcommand{\RR}{\mathbb{R}}

\makeatletter
\newcommand{\rdots}{\mathinner{\mkern1mu\lower-1\p@\vbox{\kern7\p@\hbox{.}}
\mkern2mu \raise4\p@\hbox{.}\mkern2mu\raise7\p@\hbox{.}\mkern1mu}}
\makeatother
\date{}

\begin{document}

\title{Integration in Hermite spaces of analytic functions}

\author{Christian Irrgeher, Peter Kritzer, Gunther Leobacher\\ and Friedrich Pillichshammer\thanks{The authors are supported by the Austrian Science Fund (FWF): Projects
F5506-N26 and P23389-N18 (Kritzer), F5508-N26 (Leobacher), and F5509-N26 (Irrgeher and Pillichshammer), respectively. The projects F5506-N26, F5508-N26, and F5509-N26 are part of the Special
Research Program "Quasi-Monte Carlo Methods: Theory and Applications".}}

\maketitle

\begin{abstract}
We study integration in a class of Hilbert spaces of analytic functions defined on the $\RR^s$. 
The functions are characterized by the property that their Hermite coefficients decay exponentially fast. We use Gauss-Hermite integration 
rules and show that the error of our algorithms decays exponentially fast. Furthermore, we give necessary and sufficient conditions under which 
we achieve exponential convergence with weak, polynomial, and strong polynomial tractability. %, which means that the number of function evaluations needed 
%to obtain an integration error of at most $\e$ lacks a certain disadvantageous dependence on $\e^{-1}$ and $s$. 
\end{abstract}
 
 

%%%%%%%%%%%%%%%%%%%%%%%%%%%%%%%%%%%%%%%%%%%%%%
%%%%%%%%%%%%%%%%%%%%%%%%%%%%%%%%%%%%%%%%%%%%%%

\section{Introduction}

In recent years, the theory of tractability of integration and approximation
in reproducing kernel Hilbert spaces evolved into one of the key topics
of complexity theory. Starting with the seminal work by Hickernell (see, e.g., \cite{H95})
Sloan and Wo\'zniakowski (see, e.g., \cite{SW98}), many authors have shown different types of 
tractability (or intractability) for different types of spaces. 
Here, the main attention was on spaces consisting of functions on the
$s$-dimensional unit cube. Notable exceptions are several papers by Kuo, Wasilkowski, 
Wo\'{z}niakowski and their co-authors (see, e.g., \cite{KWW06, KW12, NK14, WW02}),
where tractability of integration over unbounded domains is studied. 

Among the main reasons why the focus has been largely on function spaces over
the unit cube are the following: first of all, classical discrepancy theory
mostly considers point sets and sequences in the unit cube, and most explicit
constructions of low-discrepancy point sets and sequences have been carried
out within this framework. The reason for this, which is also
the second point in our list, is that most integrals arising in practice
can be transformed into integrals over the unit cube, at least in principle. 
A third reason is a practical one: on the unit cube we have orthogonal 
function systems, namely, the trigonometric polynomials and the Walsh functions,
that are very flexible and for which series expansion have been extensively 
studied. 

However, there are good reasons for considering spaces of functions on the
$\RR^s$. First of all, they provide the natural setup for many applications,
in particular those from mathematical finance where many models are driven
by Brownian motion.  While the corresponding integration problems can be 
transformed to ones on the unit cube, the transformed problem typically 
does not belong to any of the function classes for which tractability of
integration can be shown.

Another good reason for considering integration on the $\RR^s$, in particular
with respect to standard Gaussian measure and Lebesgue measure, is the 
symmetry of the space with respect to orthogonal transforms. % -- a feature that has no obvious counterpart on the unit cube. 
It has been found that many problems from quantitative finance benefit from orthogonal transforms when evaluated using quasi-Monte Carlo methods, see, e.g., 
\cite{ABG96,  IL12, L12,  MC96, P02, WS11}. Examples of such orthogonal transforms are provided by the Brownian bridge construction 
or the principal component construction of the paths of Brownian motion.

In this paper, we study a reproducing kernel Hilbert space with a kernel function of the form 
$$K_{r}(\bsx,\bsy)=\sum_{\bsk \in \NN_0^s} r(\bsk) H_{\bsk}(\bsx) H_{\bsk}(\bsy)\ \ \ \ \mbox{ for }\ \ \bsx,\bsy \in \RR^s,$$ 
where  $r: \NN_0^s \rightarrow \RR^+$ is a suitably chosen function, and $H_{\bsk}$ denotes, for $\bsk=(k_1,\ldots,k_s)$, 
the product of the $k_j$th Hermite polynomials. 
Function spaces of this structure have been previously studied in, e.g., \cite{IL}, 
where those spaces have been named ``Hermite spaces''. In this paper, we focus on a 
special type of the function $r(\bsk)$. We choose a fixed number 
$\omega\in (0,1)$ and choose two sequences $\bsa=\{a_j\}$ and $\bsb=\{b_j\}$ of real numbers. 
The function space $\mathcal{H}(K_{s,\bsa,\bsb,\omega})$ considered %dealt with 
in this paper is then characterized by setting 
$$r(\bsk):=\omega^{\abs{\bsk}_{\bsa,\bsb}}=\omega^{\sum_{j=1}^s a_j k_j^{b_j}}.$$ 
It can be shown that the elements of our function space are analytic and this is done in Appendix~\ref{app_a}. We are interested in studying the numerical approximation of 
$$I_s(f)=\int_{\RR^s} f(\bsx) \varphi_s(\bsx) \rd \bsx \ \ \ \mbox{ for }\ \ f \in \mathcal{H}(K_{s,\bsa,\bsb,\omega}),$$ 
where $\varphi_s$ is the density of the $s$-dimensional standard Gaussian measure. We approximate $I_s (f)$ by linear algorithms of the form 
$$A_{n,s}(f)=\sum_{i=1}^n \alpha_i f(\bsx_i) \ \ \ \mbox{ for }\ \ f \in \mathcal{H}(K_{s,\bsa,\bsb,\omega}),$$
where $\bsx_1,\ldots,\bsx_n \in \RR^s$ and $\alpha_1,\ldots,\alpha_n \in \RR$.

We then study the worst-case error of the algorithm $A_{n,s}$ , defined by 
$$e(A_{n,s},K_{s,\bsa,\bsb,\omega})=\sup_{f \in \mathcal{H}(K_{s,\bsa,\bsb,\omega}) \atop \|f\|_{K_{s,\bsa,\bsb,\omega}} 
\le 1}\abs{I_s (f)-A_{n,s} (f)},$$
where $\norm{\cdot}_{K_{s,\bsa,\bsb,\omega}}$ denotes the norm in $\mathcal{H}(K_{s,\bsa,\bsb,\omega})$. Furthermore, we define the $n$th minimal worst-case error,
$$
e(n,s)=\inf_{A_{n,s}}\ e(A_{n,s},K_{s,\bsa,\bsb,\omega}),
$$
where the infimum is extended over all linear algorithms using $n$ function evaluations.

Our first goal in this paper is to study conditions on the parameters characterizing the space 
$\mathcal{H}(K_{s,\bsa,\bsb,\omega})$ such that we obtain
exponential convergence of $e(n,s)$. By exponential convergence we mean the existence of 
a number $q\in(0,1)$ and functions $p,C,C_1:\NN\to (0,\infty)$ such that
$$
e(n,s)\le C(s)\, q^{\,(n/C_1(s))^{\,p(s)}}
\ \ \ \ \ \mbox{for
all} \ \ \ \ \ s, n\in \NN.
$$
More details on exponential convergence are given in Section~\ref{sec_ect}. 

In Theorem~\ref{thmint(u)exp}, we are going to show (matching) necessary and sufficient conditions under which we achieve exponential 
convergence, and uniform exponential convergence, which holds if $p(s)$ can be bounded uniformly in $s$. 

Our second goal is to study various notions of tractability, i.e. the asymptotic behaviour of the information complexity of integration in $\mathcal{H}(K_{s,\bsa,\bsb,\omega})$, 
$$
n(\e,s)=\min\{n\,:\, e(n,s)\le \e\},
$$
which is the minimal number $n$ of nodes needed to obtain an $\e$-approximation to $I_s (f)$, with respect to $s$ and $\e^{-1}$. To be more precise, 
we study different notions of Exponential Convergence-Tractability, which have previously been dealt with in~\cite{DKPW13,KPW14,KPW14a}. 
Roughly speaking, we mean by 
tractability that $n(\e,s)$ lacks a certain disadvantageous dependence on $s$ but also on $\e^{-1}$. We are going to derive 
necessary and sufficient conditions on 
the weight sequences $\bsa$ and $\bsb$ for three different types of tractability in Theorem~\ref{thmint(u)exp}. We remark that for two of the 
three types of tractability considered here (polynomial and strong polynomial tractability), 
our necessary and sufficient conditions match, and only for one type (weak tractability) there remains a small gap between those conditions. 

Overall, our results in Theorem~\ref{thmint(u)exp} are of a similar flavor as those in~\cite{DKPW13,DLPW11,KPW14,KPW14a}, 
but there are some major differences, most importantly that the results in those papers hold for certain analytic functions 
defined on $[0,1]^s$, and here we deal with functions defined on the $\RR^s$. We further remark that all sufficient results shown 
in this paper are based on constructive algorithms, i.e., 
we explicitly give the form of the algorithms $A_{n,s} (f)$ yielding the desired error bounds. In the case considered here, we are going to 
use Cartesian products of Gauss-Hermite rules as integration algorithms. 

The rest of the article is structured as follows. In Section~\ref{secinthermite}, we define the Hermite space $\mathcal{H}(K_{s,\bsa,\bsb,\omega})$, 
and give all details regarding the problem setting. In Section~\ref{sec_ect}, we give the precise definitions of exponential error convergence and 
we recall the notions of tractability used in this paper. In Section~\ref{secmain}, we present Theorem~\ref{thmint(u)exp}, which summarizes all 
results in this paper, and give some comments on these findings. The proof of Theorem~\ref{thmint(u)exp} is partly done in Section~\ref{necUEXP}, where 
we show lower bounds on the error of linear integration algorithms in $\mathcal{H}(K_{s,\bsa,\bsb,\omega})$, thereby obtaining necessary conditions 
for (uniform) exponential convergence and the different tractability notions. On the other hand, in Section~\ref{secGHI}, we study concrete examples 
of integration algorithms based on Gauss-Hermite rules and outline sufficient conditions for (uniform) exponential convergence and tractability. 
Finally, Appendix~\ref{app_a} contains a proof of analyticity of the elements of $\mathcal{H}(K_{s,\bsa,\bsb,\omega})$ and Appendix~\ref{app_b} shows
an example of a nontrivial function which belongs to such a Hermite space.

\section{Integration in the Hermite space}\label{secinthermite}

\subsection{Hermite polynomials and the Hermite space}

We briefly summarize some facts on \textit{Hermite polynomials}; for further details, we refer to~\cite{IL} and the references therein. 
For $k \in \NN_0$ the $k$th Hermite polynomial is given by 
$$H_k(x)=\frac{(-1)^k}{\sqrt{k!}} \exp(x^2/2) \frac{\rd^k}{\rd x^k} \exp(-x^2/2),$$ 
which is sometimes also called normalized probabilistic Hermite polynomial. Here we follow the definition given in \cite{B98}, but we remark
that there are slightly different ways to introduce Hermite polynomials (see, e.g., \cite{szeg}). For $s \ge 2$, $\bsk=(k_1,\ldots,k_s)\in \NN_0^s$, and 
$\bsx=(x_1,\ldots,x_s)\in \RR^s$ 
we define $s$-dimensional Hermite polynomials by $$H_{\bsk}(\bsx)=\prod_{j=1}^s H_{k_j}(x_j).$$ 
It is well-known (see again~\cite{B98}) that the sequence of Hermite polynomials $\{H_{\bsk}(\bsx)\}_{\bsk \in \NN_0^s}$ 
forms an orthonormal basis of the function space $L^2(\RR^s,\varphi_s)$, where $\varphi_s$ denotes the 
density of the $s$-dimensional standard Gaussian measure, 
$$\varphi_s(\bsx)=\frac{1}{(2 \pi)^{s/2}} \exp\left(-\frac{\bsx \cdot \bsx}{2}\right),$$ 
where ``$\cdot$'' is the standard inner product in $\RR^s$.  

Similar to what has been done in~\cite{IL}, we are now going to define function spaces based on Hermite polynomials. These 
spaces are Hilbert spaces with a \textit{reproducing kernel}. For details on reproducing kernel Hilbert spaces, we refer to~\cite{A50}. 

Let $r: \NN_0^s \rightarrow \RR^+$ be a summable function, i.e., $\sum_{\bsk\in\NN_0^s} r(\bsk) < \infty$. Define a kernel function 
$$K_{r}(\bsx,\bsy)=\sum_{\bsk \in \NN_0^s} r(\bsk) H_{\bsk}(\bsx) H_{\bsk}(\bsy)\ \ \ \ \mbox{ for }\ \ \bsx,\bsy \in \RR^s$$ 
and inner product 
$$\langle f,g\rangle_{K_{r}} =\sum_{\bsk \in \NN_0^s} \frac{1}{r(\bsk)} \widehat{f}(\bsk) \widehat{g}(\bsk),$$ 
where $\widehat{f}(\bsk)=\int_{\RR^s} f(\bsx) H_{\bsk}(\bsx) \varphi_s(\bsx)\rd \bsx$ is the $\bsk$th \textit{Hermite coefficient} of $f$. 
Let $\mathcal{H}(K_r)$ be the reproducing kernel Hilbert space corresponding to $K_{r}$, which we will call a Hermite space in the following. 
The norm in $\mathcal{H}(K_r)$ is given by $\| f \|_{K_{r}}^2 =\langle f , f \rangle_{K_{r}}$. From this we see that 
the functions in $\mathcal{H}(K_r)$ are characterized by the decay rate of their Hermite coefficients, which is regulated by the function $r$. 
Roughly speaking, the faster $r$ decreases as $\bsk$ grows, the faster the Hermite coefficients of the elements of $\mathcal{H}(K_r)$ decrease. 
In the recent paper~\cite{IKLP14a} the case of polynomially decreasing $r$ and in \cite{IL}, the case of polynomially decreasing $r$ as well as exponentially decreasing $r$ was considered. In this paper, we continue the work on exponentially decreasing $r$, thereby extending the results in~\cite{IL}.

To define our function $r$, we first introduce two weight sequences of real numbers, $\bsa=\{a_j\}$ and $\bsb=\{b_j\}$, where we assume that 
\begin{equation}\label{aabb}
1\le a_1\le a_2\le a_3\le\ldots\ \ \ \ \mbox{and}\ \ \ \ 1\le b_1\le b_2\le b_3\le\ldots.
\end{equation}
Furthermore, we fix a parameter $\omega\in (0,1)$. For a vector $\bsk=(k_1,\ldots,k_s)\in\NN_0^s$, we consider  
$$r(\bsk)= \omega^{\abs{\bsk}_{\bsa,\bsb}}:=\omega^{\sum_{j=1}^s a_j k_j^{b_j}}.$$
In our case we modify the notation for the kernel function to 
$$K_{s,\bsa,\bsb,\omega} (\bsx,\bsy):=\sum_{\bsk\in\NN_0^s} \omega^{\abs{\bsk}_{\bsa,\bsb}} H_{\bsk}(\bsx) H_{\bsk} (\bsy).$$
From now on, we deal with the corresponding reproducing kernel Hilbert space $\cH(K_{s,\bsa,\bsb,\omega})$. 
Our concrete choice of $r$ now decreases exponentially fast as $\bsk$ grows, which influences the smoothness of the elements in 
$\cH(K_{s,\bsa,\bsb,\omega})$. Indeed, it can be shown that functions $f \in \cH(K_{s,\bsa,\bsb,\omega})$ are analytic (see Appendix~\ref{app_a}).
We remark that reproducing kernel Hilbert spaces of a similar flavor were previously considered in~\cite{DKPW13,DLPW11,KPW14,KPW14a}, but 
the functions considered there were one-periodic functions defined on the unit cube $[0,1]^s$. Here, we study functions which are defined on the $\RR^s$, 
which is a major difference. Obviously, $\cH(K_{s,\bsa,\bsb,\omega})$ contains all polynomials on the $\RR^s$, but there are further functions 
of practical interest which belong to such spaces.  For example, it is easy to verify (see Appendix~\ref{app_b}) that $f(x_1,\ldots,x_s)=\exp(\frac{1}{\sqrt{s}}\sum_{j=1}^{s}x_j)$
is an element of the Hilbert space $\cH(K_{s,\bsa,\bsone,\omega})$ for any $\bsa$. Functions of a similar form occur in problems of financial derivative pricing, see, e.g., \cite{ll}.

\subsection{Integration}

We are interested in numerical approximation of the values of integrals 
$$I_s(f)=\int_{\RR^s} f(\bsx) \varphi_s(\bsx) \rd \bsx \ \ \ \mbox{ for }\ \ f \in \mathcal{H}(K_{s,\bsa,\bsb,\omega}).$$ 
Without loss of generality, see, e.g., \cite[Section~4.2]{NW08} or \cite{TWW88}, we can restrict ourselves to 
approximating $I_s(f)$ by means of {\it linear algorithms} of the form 
\begin{equation}\label{linAlg}
A_{n,s}(f)=\sum_{i=1}^n \alpha_i f(\bsx_i) \ \ \ \mbox{ for }\ \ f \in \mathcal{H}(K_{s,\bsa,\bsb,\omega})
\end{equation}
with integration nodes $\bsx_1,\ldots,\bsx_n \in \RR^s$ and weights $\alpha_1,\ldots,\alpha_n \in \RR$. 
For $f \in \mathcal{H}(K_{s,\bsa,\bsb,\omega})$ let 
$${\rm err}(f):=I_s(f) -A_{n,s}(f).$$ 
The {\it worst-case error} of the algorithm $A_{n,s}$ is then defined as the worst performance of $A_{n,s}$ over the unit ball 
of $\mathcal{H}(K_{s,\bsa,\bsb,\omega})$, i.e., 
$$e(A_{n,s},K_{s,\bsa,\bsb,\omega})=\sup_{f \in \mathcal{H}(K_{s,\bsa,\bsb,\omega}) \atop \|f\|_{K_{s,\bsa,\bsb,\omega}} 
\le 1}\left|{\rm err}(f)\right|.$$ 
Using standard arguments (see, e.g., \cite[Theorem~3.5]{DKS} or \cite[Proposition~2.11]{DP10}) from the theory of 
numerical integration in reproducing kernel Hilbert spaces we obtain 
\begin{eqnarray*}
e^2(A_{n,s},K_{s,\bsa,\bsb,\omega})&= & \int_{\RR^s} 
\int_{\RR^s} K_{s,\bsa,\bsb,\omega}(\bsx,\bsy)\varphi_s(\bsx)\varphi_s(\bsy)\rd \bsx \rd \bsy\\
&& -2 \sum_{i=1}^n \alpha_i\int_{\RR^s} K_{s,\bsa,\bsb,\omega}(\bsx,\bsx_i) \varphi_s(\bsx)\rd \bsx\\
&& +\sum_{i=1}^n \sum_{j=1}^n \alpha_i \alpha_j K_{s,\bsa,\bsb,\omega}(\bsx_i,\bsx_j). 
\end{eqnarray*}
Inserting the kernel function yields
\begin{align}\label{formula_wc-error}
e^2(A_{n,s},K_{s,\bsa,\bsb,\omega})= & \left(-1+\sum_{i=1}^n \alpha_i\right)^2+\sum_{\bsk \in \NN_0^s \setminus \{\bszero\}} \omega^{|\bsk|_{\bsa,\bsb}}\left(\sum_{i=1}^n \alpha_i H_{\bsk}(\bsx_i)\right)^2.
\end{align}

Let $e(n,s)$ be the {\it $n$th minimal worst-case error},
$$
e(n,s)=\inf_{A_{n,s}}\ e(A_{n,s},K_{s,\bsa,\bsb,\omega}),
$$
where the infimum is extended over all linear algorithms of the form \eqref{linAlg}, i.e., 
over all nodes $\bsx_1,\ldots,\bsx_n$ and  all weights $\alpha_1,\ldots,\alpha_n$.

For $n=0$, the best we can do is to approximate $I_s(f)$ simply by zero, and
$$
e(0,s)=\|I_s\|=\int_{\RR^s} 
\int_{\RR^s} K_{s,\bsa,\bsb,\omega}(\bsx,\bsy)\varphi_s(\bsx)\varphi_s(\bsy)\rd \bsx \rd \bsy=1\ \ \ \ \ \mbox{for all}\ \ \ \ \ s\in \NN.
$$
Hence, the integration problem is well normalized for all $s$.

For $\e\in(0,1)$, we define the {\it information complexity} of integration
$$
n(\e,s)=\min\{n\,:\, e(n,s)\le \e\}
$$
as the minimal number of function values needed to obtain an
$\e$-approximation.

\section{Exponential convergence and tractability}\label{sec_ect}

Since the functions belonging to the function space $\mathcal{H}(K_{s,\bsa,\bsb,\omega})$ are very smooth, 
it is natural to expect that, by using suitable algorithms, we should be able to obtain an integration error that converges to zero very quickly as $n$ increases. Indeed, 
what we would like to achieve is exponential convergence of the integration error, and we first define this type of convergence in detail.

We say that we achieve  \emph{exponential convergence} (EXP)
for $e(n,s)$ if
there exists a number $q\in(0,1)$ and
functions $p,C,C_1:\NN\to (0,\infty)$ such that
\begin{equation}\label{exrate}
e(n,s)\le C(s)\, q^{\,(n/C_1(s))^{\,p(s)}}
\ \ \ \ \ \mbox{for
all} \ \ \ \ \ s, n\in \NN.
\end{equation}
We refer to \cite{DKPW13, KPW14, KPW14a} for detailed information on the notion of exponential convergence.
If \eqref{exrate} holds we would like to find the largest possible
rate $p(s)$ of exponential convergence
defined as
%$$
%p^*(s)=\sup\{\,p(s)\ :\ \ p(s)\ \ \mbox{satisfies \eqref{exrate}}\,\}.
%$$
$$
p^*(s) = \sup \{\,p\in (0,\infty): \exists \, C, C_1 \in (0, \infty) \, \mbox{ such that } \, \forall n \in \NN: e(n, s) \le C q^{(n/C_1)^p} \}.
$$

We say that we achieve \emph{uniform exponential convergence} (UEXP)
for $e(n,s)$ if the function $p$
in \eqref{exrate} can be taken as a constant function, i.e.,
$p(s)=p>0$ for all $s\in\NN$. Similarly, let
%$$
%p^*=\sup\{\,p\ :\ \ p(s)=p>0\ \ \mbox{satisfies \eqref{exrate} for all
%$s\in\NN$}\,\}
%$$
$$
p^* = \sup \{ p \in (0, \infty): \exists \, C, C_1: \NN \rightarrow (0, \infty) \, \mbox{ such that } \, \forall n, s \in \NN: e (n, s) \le C q^{(n/C_1)^p} \} 
$$
denote the largest rate of uniform exponential convergence.
\vskip 1pc
Exponential convergence implies that asymptotically, with
respect to $\e$ tending to zero, we need $\mathcal{O}(\log^{1/p(s)}
\e^{-1})$ function evaluations to compute an
$\e$-approximation to an integral.
However, it is not
clear how long we have to wait to see this nice asymptotic
behavior especially for large $s$. This, of course, depends on
how $C(s),C_1(s)$ and $p(s)$ depend on $s$, and this is the subject of tractability. The following
tractability notions were already introduced in \cite{DKPW13, DLPW11, KPW14} but the 
corresponding nomenclature was introduced later in \cite{KPW14a}. We say that we have:
\begin{itemize}
\item[(a)] \emph{Exponential Convergence-Weak Tractability (EC-WT)} if
$$
\lim_{s+\log\,\e^{-1}\to\infty}\frac{\log\
  n(\varepsilon,s)}{s+\log\,\e^{-1}}=0\ \ \ \ \ \mbox{with}\ \ \
 \log\,0=0\ \ \mbox{by convention}.
$$
\item[(b)] \emph{Exponential Convergence-Polynomial Tractability
    (EC-PT)} if there exist non-negative
  numbers $c,\tau_1,\tau_2$ such that
$$
n(\varepsilon,s)\le
c\,s^{\,\tau_1}\,(1+\log\,\e^{-1})^{\,\tau_2}\ \ \ \ \ \mbox{for all}\ \ \ \
s\in\NN, \ \e\in(0,1).
$$
\item[(c)]
\emph{Exponential Convergence-Strong Polynomial Tractability (EC-SPT)}
if there exist non-negative
  numbers $c$ and $\tau$  such that
$$
n(\varepsilon,s)\le
c\,(1+\log\,\e^{-1})^{\,\tau}\ \ \ \ \ \mbox{for all}\ \ \ \
s\in\NN, \ \e\in(0,1).
$$
The exponent $\tau^*$ of EC-SPT is defined as
the infimum of $\tau$ for which EC-SPT holds, i.e.,
$$ \tau^* = \inf \{ \tau \ge 0: \exists \, c \in [0, \infty) \mbox{ such that } n (\varepsilon, s) \le c ( 1+\log \varepsilon^{-1})^\tau \quad \forall s \in \NN, \varepsilon \in (0,1) \}.$$
\end{itemize}

Let us give some comments on these definitions.
First, we remark that the use of the prefix EC
(exponential convergence) in (a)--(c)
is motivated by the fact that EC-PT (and therefore also
EC-SPT) implies UEXP (see~\cite{DKPW13,KPW14a}). Also EC-WT implies that $e(n,s)$
converges to zero faster than any power of $n^{-1}$ as $n$ goes to infinity (see \cite{KPW14a}), i.e., 
$$\lim_{n \rightarrow \infty} n^{\alpha} e(n,s)=0 \ \ \ \mbox{ for all }\ \ \ \alpha \in \RR^+ \ \ \ \mbox{and all} \ \ \ s \in \NN.$$

Furthermore we note, as in \cite{DKPW13, DLPW11},
that if \eqref{exrate} holds then
\begin{equation}\label{exrate2}
n(\e,s)
\le \left\lceil C_1(s) \left(\frac{\log C(s) +
    \log \e^{-1}}{\log q^{-1}}\right)^{1/p(s)}\right\rceil
\ \ \ \ \ \mbox{for all}\ \ \ s\in \NN\ \ \mbox{and}\ \ \e\in (0,1).
\end{equation}
Conversely, if~\eqref{exrate2} holds then
$$
e(n+1,s)\le C(s)\, q^{\,(n/C_1(s))^{\,p(s)}}\ \
\ \ \ \mbox{for all}\ \ \ s,n\in \NN.
$$
This means that~\eqref{exrate} and~\eqref{exrate2} are practically
equivalent. Note that $1/p(s)$ determines the power of $\log\,\e^{-1}$
in the information complexity,
whereas $\log\,q^{-1}$ only affects the multiplier of $\log^{1/p(s)}\e^{-1}$.
From this point of view, $p(s)$ is more
important than $q$. 

EC-WT means that we rule out the cases for which
$n(\e,s)$ depends exponentially on $s$ and $\log\,\e^{-1}$, whereas EC-PT means that the information complexity 
depends at most polynomially on $s$ and $\log\,\e^{-1}$. If we even have EC-SPT this translates into $n(\e,s)$ depending 
at most polynomially on $\log\,\e^{-1}$, but not on $s$ anymore. 

We remark that, in many papers, tractability is studied for problems where we do not have exponential, but usually polynomial, error convergence. 
For this kind of problems, tractability has been defined by studying how the information complexity depends on $s$ and $\e^{-1}$ (for a detailed survey 
of such results, we refer to~\cite{NW08}--\cite{NW12}). With the notions 
of EC-tractability considered in~\cite{DKPW13, DLPW11, KPW14, KPW14a} and in the present paper, however, we study how the information complexity depends 
on $s$ and $\log\,\e^{-1}$.

\section{The main results}\label{secmain}

In this section we summarize the main results of our paper. The following theorem gives necessary and sufficient conditions 
on the weight sequences $\bsa$ and $\bsb$ for (uniform) exponential convergence, and the notions of EC-WT, EC-PT, and EC-SPT.

\begin{theorem}\label{thmint(u)exp}
Consider integration defined over the Hermite space
$H(K_{s,\bsa,\bsb,\omega})$ with weight sequences $\bsa$ and
$\bsb$ satisfying~\eqref{aabb}.
\begin{enumerate}
\item\label{intexp}
EXP holds for all $\bsa$ and $\bsb$ considered, and
$$
p^{*}(s)=\frac{1}{B(s)} \ \ \ \ \
\mbox{with}\ \ \ \ \ B(s):=\sum_{j=1}^s\frac1{b_j}.
$$
\item\label{intuexp}
The following assertions are equivalent:
\renewcommand{\labelenumii}{(\alph{enumii})}
\begin{enumerate}
\item The sequence $\bsb=\{b_j\}_{j \ge 1}$ is summable, i.e., 
$$
B:=\sum_{j=1}^\infty\frac1{b_j}<\infty;
$$
\item we have UEXP;
\item we have EC-PT;
\item we have EC-SPT.
\end{enumerate}
If one of the assertions holds then $p^*=1/B$ and the exponent $\tau^{\ast}$ of EC-SPT is $B$.
\item\label{wtnec} EC-WT implies that $\lim_{j \rightarrow \infty} a_j 2^{b_j} =\infty$.
\item\label{wtsuff} A sufficient condition for EC-WT is that there exist $\eta>0$ and $\beta>0$ such that $$a_j 2^{b_j} \ge \beta j^{1+\eta}\ \ \ \ \mbox{ for all }\ \ j \in \NN.$$
\end{enumerate}
\end{theorem}

Let us give some remarks on Theorem~\ref{thmint(u)exp}. Item~\ref{intexp} states that we always achieve exponential convergence, 
independently of the choice of the sequences 
$\bsa$ and $\bsb$. In particular, the best rate $p^{*}(s)$ is given by $1/B(s)$. 
As all $b_j$ are positive, this implies that $p^* (s)$ decreases with $s$, and if $B(s)$ diverges, $p^* (s)$ tends to zero. If 
$\bsb$ is a constant function consisting only of ones, we get the lowest rate of exponential convergence, namely $1/s$. 

The second item in Theorem~\ref{thmint(u)exp} states that the condition $B<\infty$ and the notions of UEXP, EC-PT, and EC-SPT are all equivalent. In particular, this implies that EC-PT and EC-SPT hold if and only 
if we have UEXP. Hence we can say that we practically know everything, including $p^*$ and $\tau^*$, about 
UEXP, EC-PT and EC-SPT. Note, furthermore, that the choice of $\bsa$ has no influence whatsoever on Item~\ref{intuexp}. The situation is different 
for the results in~\cite{DKPW13, KPW14, KPW14a}, where the $a_j$ have to grow exponentially fast in order to obtain UEXP. 

Regarding Items~\ref{wtnec} and~\ref{wtsuff}, we observe that the situation for EC-WT is quite different from that for EC-PT and EC-SPT. 
First of all, note that the sequence $\bsa$ plays an important role with respect 
to EC-WT as opposed to EC-PT and EC-SPT. We can have EC-WT if the elements $a_j$ of $\bsa$ increase sufficiently fast 
even if $\bsb$ is a constant function. This also implies that for EC-WT it is relevant to distinguish 
between EC-WT with UEXP and EC-WT without UEXP. If we have UEXP, then we automatically have EC-WT, but the converse 
does in general not hold. Note furthermore that there is a gap between the necessary and sufficient 
conditions for EC-WT. Indeed, 
we tried hard to close this gap but it seems that the methods currently at hand are not powerful enough, so this problem remains open for future research. We conjecture 
that the weaker condition, $\lim_{j \rightarrow \infty} a_j 2^{b_j} =\infty$, is also sufficient for EC-WT. 

Finally, we remark that our assumptions on $\bsa$ and $\bsb$ are slightly more restrictive than those in~\cite{DKPW13,KPW14a}. Indeed, 
our restrictions that both sequences are non-decreasing and bounded  from below by 1 are used for 
deriving the powerful upper bounds on the integration error of Gauss-Hermite rules in Section~\ref{secGHI}. The question of how to show similar results
for more general choices of $\bsa$ and $\bsb$ is left open for future research. 

We will see in the proof of Theorem~\ref{thmint(u)exp} that EXP, UEXP, EC-WT, EC-PT and EC-SPT, respectively, are all achieved by Cartesian products of Gauss-Hermite rules (see Theorem~\ref{thm_suff_uexp} and \ref{thm_suff_ecspt} and the proof of Theorem~\ref{thm_suff_ecwt}). 

The proof of the Theorem~\ref{thmint(u)exp} is organized as follows: In Section~\ref{necUEXP} we show that UEXP implies $\sum_{j=1}^\infty\frac1{b_j}<\infty$ (see Theorem~\ref{thmnecUEXP}). 
In Section~\ref{secGHI} we show that we always have EXP and we show that  $\sum_{j=1}^\infty\frac1{b_j}<\infty$ implies UEXP (see Theorem~\ref{thm_suff_uexp}). 
Next we show that $\sum_{j=1}^\infty\frac1{b_j}<\infty$ implies EC-SPT (see Theorem~\ref{thm_suff_ecspt}). 
The remainig part of the equivalence in the second item is obvious, since EC-PT implies UEXP (as mentioned in Section~\ref{sec_ect}) and hence: 
$$\mbox{EC-PT} \Rightarrow \mbox{UEXP} \Rightarrow \sum_{j=1}^{\infty} \frac{1}{b_j}< \infty \Rightarrow \mbox{EC-SPT}\Rightarrow \mbox{EC-PT}.$$

The necessary condition for EC-WT will be shown at the end of Section~\ref{necUEXP} (see Theorem~\ref{thm_nec_wt}) 
and the sufficient condition at the end of Section~\ref{secGHI} (see Theorem~\ref{thm_suff_ecwt}).

\section{Lower bounds}\label{necUEXP}

In this section we prove the necessity of the condition for UEXP from Theorem~\ref{thmint(u)exp}. The procedure to show the following 
results is inspired by what is done in~\cite{DLPW11, KPW14}.

First we require the following two lemmas.
\begin{lem}\label{lemHermpoly}
For $k,l \in \NN_0$ we have $$\int_{\RR} H_k(x) H_l(x) \varphi(x) \rd x = \left\{
\begin{array}{ll}
1 & \mbox{ if } k=l,\\
0 & \mbox{ if } k \not=l. 
\end{array}\right.
$$
For $k,l,m \in \NN$ we have $$\int_{\RR} H_k(x) H_l(x) H_m(x) \varphi(x) \rd x=\left\{
\begin{array}{ll}
 \frac{\sqrt{k! l! m!}}{(s-k)! (s-l)! (s-m)!} & \mbox{ if } k+l+m=2 s \mbox{ and } k,l,m \le s,\\
0 & \mbox{ in all other cases}.
\end{array}\right.$$
\end{lem}

\begin{proof}
The first identity follows from the orthogonality of Hermite polynomials. The second one follows from \cite[p.~390]{szeg}. 
\end{proof}

\begin{lem}\label{lemtechnical}
Let $t \in \NN$. For $k,l \in \{0,1,\ldots,t\}$ and $m \in\{0,1,\ldots,2 t\}$ such that $k+l+m=2s$ and $k,l,m \le s$ we have 
$$ \frac{\sqrt{k! l! m!}}{(s-k)! (s-l)! (s-m)!} \le 4^t.$$
\end{lem}

\begin{proof} 
We use the notation $$a_{k,l,m}:=\frac{\sqrt{k! l! m!}}{(s-k)! (s-l)!
(s-m)!}=\frac{\sqrt{k! l! m!}}{\left(\frac{l+m-k}{2}\right)!
\left(\frac{k+m-l}{2}\right)! \left(\frac{k+l-m}{2}\right)!}.$$ 

Note that $\frac{l+m-k}{2}+\frac{k+l-m}{2}=l$,
$\frac{k+m-l}{2}+\frac{k+l-m}{2}=k$ and $
\frac{l+m-k}{2}+\frac{k+m-l}{2}=m$, such that 
\[
a_{k,l,m}^2={k \choose \frac{k+m-l}{2}}{l \choose \frac{l+m-k}{2}}
{m \choose \frac{k+m-l}{2}}\,.
\]
Now since ${k \choose \frac{k+m-l}{2}} \le \sum_{j=0}^k {k \choose j}=2^k$, and
analog estimates for the other binomial coefficients hold, we have
\[
a_{k,l,m}^2 \le 2^k2^l2^m=2^{k+l+m}.
\]
Using the assumptions on $k,j,m$, we get
\[
a_{k,l,m}^2 \le 2^{t+t+2t}=2^{4t},
\]
i.e., $a_{k,l,m} \le 2^{2t}=4^t$.
\end{proof}

Using the previous lemmas, we derive the following general lower bound on the $n$th minimal worst-case error. 

\begin{theorem}\label{lbound}
The $n$th minimal worst-case error satisfies
\begin{equation}\label{lowerbound}
e(n,s)\ge \frac{\omega^{\sum_{j=1}^s a_j (2 t_j)^{b_j}}}{\prod_{j=1}^s (4^{t_j} 2 (t_j+1)^2)}\ \ \ \ \mbox{ for all }
\ \ \ \ n< \prod_{j=1}^s(t_j+1).
\end{equation}
\end{theorem}
\begin{proof}
Let $\mathcal{A}_s=\prod_{j=1}^s \{0,1,\ldots,t_j\}$ with $t_j \in \NN$ for $j=1,2,\ldots,s$. For $\bsh\in\mathcal{A}_s$, we denote the components of $\bsh$ by $h_1,\ldots,h_s$. We have $\abs{\mathcal{A}_s}=\prod_{j=1}^s(t_j+1)$.
 
For $n< \abs{\mathcal{A}_s}$ take an arbitrary linear algorithm $A_{n,s}(f)=\sum_{m=1}^n \alpha_m f(\bsx_m)$. Define 
$$
g(\bsx)=\sum_{\bsh\in \mathcal{A}_{s}}b_{\bsh}\,H_{\bsh}(\bsx)\ \ \
\mbox{for all}\ \ \ \bsx\in\RR^s
$$
such that $g(\bst_m)=0$ for all $m=1,2,\dots,n$. Since we have $n$ 
homogeneous linear equations and $|\mathcal{A}_{s}|>n$ unknowns $b_{\bsh}$, 
there exists a
nonzero vector of such $b_{\bsh}$'s, 
and we can normalize $b_{\bsh}$'s by assuming that
$$
\max_{\bsh\in \mathcal{A}_{s}}|b_{\bsh}|=b_{\bsh^*}=1 \ \ \ \mbox{for some}\ \ \ 
\bsh^*\in \mathcal{A}_{s}.
$$ 
Define the function
\begin{align*}
f(\bsx)= & c\, H_{\bsh^{\ast}}(\bsx) \,g(\bsx)
\end{align*}
with a positive $c$ which we determine such that $\|f\|_{K_{s,\bsa,\bsb,\omega}} \le 1$. To this end we 
need to estimate the Hermite coefficients of $f$. We have
\begin{align*}
|\widehat{f}(\bsl)| = & \left|\int_{\RR^s} f(\bsx)  H_{\bsl}(\bsx) \varphi_s(\bsx) \rd \bsx\right|\\
\le & c \, \sum_{\bsh\in\mathcal{A}_{s}} |b_{\bsh}| \left|\int_{\RR^s} H_{\bsh}(\bsx) H_{\bsh^{\ast}}(\bsx) H_{\bsl}(\bsx) \varphi_s(\bsx) \rd \bsx\right| \\
\le & c \, \sum_{\bsh\in\mathcal{A}_{s}}  \prod_{j=1}^s \left|\int_{\RR} H_{h_j}(x) H_{h_j^{\ast}}(x) H_{l_j}(x) \varphi(x) \rd x\right|.
\end{align*}
Hence we have
\begin{align*}
\|f\|_{K_{s,\bsa,\bsb,\omega}}^2 = & \sum_{\bsl \in \NN_0^s} |\widehat{f}(\bsl)|^2 \omega^{-|\bsl|_{\bsa,\bsb}}\\
\le & c^2 \sum_{\bsl \in \NN_0^s} \left(\sum_{\bsh\in\mathcal{A}_{s}}  \prod_{j=1}^s \left|\int_{\RR} H_{h_j}(x) H_{h_j^{\ast}}(x) H_{l_j}(x) \varphi(x) \rd x\right| \right)^2 \omega^{-|\bsl|_{\bsa,\bsb}}\\
\le & c^2 \left(\sum_{\bsl \in \NN_0^s} \sum_{\bsh\in\mathcal{A}_{s}}  \prod_{j=1}^s \left|\int_{\RR} H_{h_j}(x) H_{h_j^{\ast}}(x) H_{l_j}(x) \varphi(x) \rd x\right|  \omega^{-|\bsl|_{\bsa,\bsb}} \right)^2\\
= & c^2 \left(\prod_{j=1}^s \sum_{l_j=0}^{\infty} \omega^{-a_j l_j^{b_j}} \sum_{h_j=0}^{t_j}   \left|\int_{\RR} H_{h_j}(x) H_{h_j^{\ast}}(x) H_{l_j}(x) \varphi(x) \rd x\right|\right)^2.
\end{align*}
Since $h_j,h_j^{\ast} \in \{0,1,\ldots,t_j\}$ it follows that 
$$\int_{\RR} H_{h_j}(x) H_{h_j^{\ast}}(x) H_{l_j}(x) \varphi(x) \rd x=0$$ whenever $l > 2 t_j$. Therefore and with Lemma~\ref{lemtechnical},
\begin{align*}
\lefteqn{\sum_{l_j=0}^{\infty} \omega^{-a_j l_j^{b_j}} \sum_{h_j=0}^{t_j}   \left|\int_{\RR} H_{h_j}(x) H_{h_j^{\ast}}(x) H_{l_j}(x) \varphi(x) \rd x\right|}\\
& =  \sum_{l_j=0}^{2 t_j} \omega^{-a_j l_j^{b_j}} \sum_{h_j=0}^{t_j}   \left|\int_{\RR} H_{h_j}(x) H_{h_j^{\ast}}(x) H_{l_j}(x) \varphi(x) \rd x\right|\\
& \le  4^{t_j} (t_j+1) \sum_{l_j=0}^{2 t_j} \omega^{-a_j l_j^{b_j}} \\
& \le  4^{t_j} (t_j+1) (2t_j+1) \omega^{- a_j (2 t_j)^{b_j}}.
\end{align*}
This gives 
\begin{align*}
\|f\|_{K_{s,\bsa,\bsb,\omega}}\le & c \ \omega^{-\sum_{j=1}^s a_j (2 t_j)^{b_j}} \prod_{j=1}^s (4^{t_j} 2 (t_j+1)^2)
\end{align*}
and hence we can take
$$
c= \frac{\omega^{\sum_{j=1}^s a_j (2 t_j)^{b_j}}}{\prod_{j=1}^s (4^{t_j} 2 (t_j+1)^2)}
$$
in order to achieve $\|f\|_{K_{s,\bsa,\bsb,\omega}} \le 1$.
Note that $f(\bst_m)=0$ and this implies that $A_{n,s}(f)=0$.
Furthermore, $\int_{\RR^s}f(\bsx) \varphi_s(\bsx)\rd \bsx=c\,b_{\bsh^*}=c$. Hence,
$$
e(A_{n,s},K_{s,\bsa,\bsb,\omega})\ge \left|\int_{\RR^s}f(\bsx) \varphi_s(\bsx)\rd \bsx-A_{n,s}(f)\right|=\int_{\RR^s}f(\bsx) \varphi_s(\bsx)\rd \bsx=c.
$$ 
Since  this holds for all $a_m$ and $\bst_m$, we conclude that
$e(n,s)\ge c$, as claimed.
\end{proof}

From Theorem~\ref{lbound}, we derive the following theorem, which implies the necessary conditions for UEXP stated in Item~\ref{intuexp} of 
Theorem~\ref{thmint(u)exp}.

\begin{theorem}\label{thmnecUEXP}
Assume that we have UEXP, i.e., there exist numbers $q\in (0,1)$, $p>0$, 
and functions $C,C_1:\NN\To (0,\infty)$ such that
\begin{equation}\label{eqexpconv}
 e(n,s)\le C(s) q^{(n/ C_1 (s))^p}\ \ \forall n,s\in\NN.
\end{equation}
Then for arbitrary $s\in\NN$ and all $\bst=(t_1,\ldots,t_s)\in\NN^s$ with $\norm{\bst}_{s,\infty}:=\max_{1\le j\le s} t_j$ 
tending to infinity we have
$$\liminf_{\norm{\bst}_{s,\infty}\To\infty} \frac{\sum_{j=1}^s a_j (2t_j)^{b_j}}{\prod_{j=1}^s (1+t_j)^p}\ge \frac{\log q^{-1}}{\log\omega^{-1}}
C_1 (s)^{-p}>0.$$
In particular, this implies that 
$$ B:=\sum_{j=1}^\infty \frac{1}{b_j}<\infty\ \mbox{and}\ p\le\frac{1}{B},$$
independently of $\bsa$ and $\omega$.
\end{theorem}

\begin{proof}
Assume that \eqref{eqexpconv} holds. Let $\bst=(t_1,\ldots,t_s)\in\NN^s$ and choose $n=-1+\prod_{j=1}^s (t_j+1)$. Then \eqref{eqexpconv} and Theorem~\ref{lbound} imply 
$$C(s) q^{(n/ C_1 (s))^p}\ge  \omega^{\sum_{j=1}^s a_j (2t_j)^{b_j}}\prod_{j=1}^s (4^{t_j} 2 (t_j+1))^{-1},$$
which implies
$ C(s) \ge
\frac{\omega^{\sum_{j=1}^s a_j (2t_j)^{b_j}}\prod_{j=1}^s (4^{t_j} 2(t_j+1))^{-1}}{q^{(n/ C_1 (s))^p}}$ and therefore
\begin{align*}
\log C(s)& \ge
\left(\log \left( \omega^{\sum_{j=1}^s a_j (2t_j)^{b_j}}\right) +\log\left(\prod_{j=1}^s (4^{t_j} 2 (t_j+1))^{-1}\right) 
 -\log \left(q^{(n/ C_1 (s))^p}\right) \right)\\
&=\left(-\sum_{j=1}^s a_j (2t_j)^{b_j} \log \omega^{-1}-\log 4^{\sum_{j=1}^s t_j}-\sum_{j=1}^s \log (2(t_j+1)) + \left(\frac{n}{C_1 (s)}\right)^p\log q^{-1}\right)\\
& \ge \left(-\sum_{j=1}^s a_j (2t_j)^{b_j} \log (4\omega^{-1})-\sum_{j=1}^s \log (2(t_j+1)) + \left(\frac{n}{C_1 (s)}\right)^p\log q^{-1}\right).
\end{align*}
This implies 
$$\log C(s) +  \sum_{j=1}^s a_j (2t_j)^{b_j} \log(4 \omega^{-1})+\sum_{j=1}^s \log (2(t_j+1))\ge n^p \log (q^{-1}) C_1 (s)^{-p},$$
or, equivalently,
$$\log C(s) +  \sum_{j=1}^s a_j (2t_j)^{b_j} \log(4 \omega^{-1})+\sum_{j=1}^s \log (2(t_j+1))\ge 
\left(\prod_{j=1}^s (t_j+1)-1\right)^p \log (q^{-1}) C_1 (s)^{-p},$$
which, in turn, is equivalent to
$$
\sum_{j=1}^s a_j (2t_j)^{b_j} +\frac{\log C(s)+\sum_{j=1}^s \log (2(t_j+1))}{\log(4 \omega^{-1})} \ge 
\left(\prod_{j=1}^s (t_j+1)-1\right)^p \frac{\log q^{-1}}{\log(4 \omega^{-1})} C_1 (s)^{-p}.
$$
This implies 
\begin{multline}\label{eqB1}
\frac{\sum_{j=1}^s a_j (2t_j)^{b_j}}{\prod_{j=1}^s (t_j+1)^p} 
+\frac{\log C(s)+\sum_{j=1}^s \log (2(t_j+1))}{(\prod_{j=1}^s (t_j+1)^p) \log(4\omega^{-1})} \ge \\
\ge \left(1-\frac{1}{\prod_{j=1}^s (t_j+1)}\right)^p \frac{\log q^{-1}}{\log(4 \omega^{-1})} C_1 (s)^{-p}.
\end{multline}
For fixed $s$, when $\norm{\bst}_{s,\infty}=\max_{1\le j\le s} t_j\To\infty$, then the second term of the left hand side of \eqref{eqB1} 
goes to zero, and it follows that
\begin{align}\label{eqliminf}
\liminf_{\norm{\bst}_{s,\infty}\To\infty} \frac{\sum_{j=1}^s a_j (2t_j)^{b_j}}{\prod_{j=1}^s (t_j+1)^p}\ge \frac{\log q^{-1}}{\log(4 \omega^{-1})}
C_1 (s)^{-p}>0.
\end{align}

For a positive number $t$ take now 
$$t_j:=\left\lceil t^{1/b_j}\right\rceil\ \mbox{ for all } \ j=1,2,\ldots,s.$$
Clearly, $\lim_{t\To\infty} \left\lceil t^{1/b_j}\right\rceil /t^{1/b_j}=1$. Then, for $t$ tending to infinity, we have 
\begin{eqnarray}\label{eqB2}
 \frac{\sum_{j=1}^s a_j (2t_j)^{b_j}}{\prod_{j=1}^s (t_j+1)^p}&=&
\frac{\sum_{j=1}^s a_j (2\left\lceil t^{1/b_j}\right\rceil)^{b_j}}{\prod_{j=1}^s (\left\lceil t^{1/b_j}\right\rceil+1)^p}\nonumber\\
&=&\frac{\sum_{j=1}^s a_j (2\left\lceil t^{1/b_j}\right\rceil)^{b_j}t/((t^{1/b_j})^{b_j})}
{\prod_{j=1}^s (\left\lceil t^{1/b_j}\right\rceil+1)^p \prod_{j=1}^s (t^{1/b_j})^p /\prod_{j=1}^s (t^{1/b_j})^p}\nonumber\\
&=& t^{1-p\sum_{j=1}^s b_j^{-1}}\frac{\sum_{j=1}^s a_j (2\left\lceil t^{1/b_j}\right\rceil/t^{1/b_j})^{b_j}}
{\prod_{j=1}^s (\left\lceil t^{1/b_j}\right\rceil /t^{1/b_j}+t^{-1/b_j})^p}.
\end{eqnarray}
Now if $t\To\infty$, then 
$$\frac{\sum_{j=1}^s a_j (2\left\lceil t^{1/b_j}\right\rceil/t^{1/b_j})^{b_j}}
{\prod_{j=1}^s (\left\lceil t^{1/b_j}\right\rceil /t^{1/b_j}+t^{-1/b_j})^p}$$
tends to $\sum_{j=1}^s a_j 2^{b_j}$. We know from \eqref{eqliminf} that expression \eqref{eqB2} is bounded away from $0$, so we must have $p\sum_{j=1}^s b_j^{-1}\le 1$. This holds for all $s$. Hence, 
for $s$ tending to infinity, we conclude that 
$$p\sum_{j=1}^\infty \frac{1}{b_j}=pB \le 1,$$
which finishes the proof of the theorem.
\end{proof}

Finally we have the following theorem providing necessary conditions for EC-WT and thus implies
Item~\ref{wtnec} of Theorem~\ref{thmint(u)exp}.

\begin{theorem}\label{thm_nec_wt}
 Assume that we have EC-WT, i.e., we have 
$$
\lim_{s+\log\,\e^{-1}\to\infty}\frac{\log\
  n(\varepsilon,s)}{s+\log\,\e^{-1}}=0.$$ 
Then it follows that $\lim_{j \rightarrow \infty} a_j 2^{b_j}=\infty$.
\end{theorem}
\begin{proof}
Assume that $(a_j 2^{b_j})_{j \ge 0}$ is bounded, say $a_j 2^{b_j}\le A<\infty$ 
for all $j \in \NN$. {}From setting $t_1 = t_2 = \ldots = 1$ in Theorem~\ref{lbound} it follows that for all $n < 2^s$ we have 
$$
e(n,s) 
\ge 64^{-s} \,\omega^{\sum_{j=1}^s a_j2^{b_j}}
\ge 64^{-s}\,\omega^{A s}= \eta^s,
$$ 
where $\eta:= \omega^{A}/64 \in (0,1)$. 
Hence, for $\varepsilon=\eta^s/2$ we have 
$e(n,s)>\e$ for all $n<2^s$. This implies that 
$n(\varepsilon,s)\ge 2^s$ and  
$$\frac{\log n(\varepsilon,s)
}{s+\log \varepsilon^{-1}} 
\ge \frac{s \log 2}{s+\log 2 +s \log \eta^{-1}} 
\rightarrow  \frac{\log 2}{1+\log \eta^{-1}} >0
\quad \mbox{as\ \ $s \rightarrow \infty$}. 
$$ 
Thus we do not have EC-WT. 
\end{proof}

\section{Gauss-Hermite integration}\label{secGHI}

To show the sufficiency of the conditions in Theorem~\ref{thmint(u)exp} we use Cartesian products of 
Gauss-Hermite rules of different order. In preparation for the general case we first consider the one-dimensional case. 
We remark that Gauss-Hermite rules for univariate integration in a different type of reproducing kernel Hilbert spaces was recently studied 
in~\cite{KW12}.

\subsection{The one-dimensional case}

Throughout this section we omit the index for the dimension (which is one) for the sake of simplicity.

A {\it Gauss-Hermite rule of order $n$} is a linear integration rule $A_{n}$ of the form $A_n(f)=\sum_{i=1}^n \alpha_i f(x_i)$ 
that is exact for all polynomials of degree less then $2n$, i.e. 
$$\int_{\RR} p(x) \varphi(x) \rd x=\sum_{i=1}^n \alpha_i p(x_i)$$ 
for all $p \in \RR[x]$ with $\deg(p)<2n$. The nodes $x_1,\ldots,x_n \in \RR$ are exactly the zeros of the $n$th Hermite polynomial 
$H_n$ and the weights are given by $$\alpha_i=\frac{1}{n H_{n-1}^2(x_i)}.$$ We remark that the weights $\alpha_i$ are all positive and that 
\begin{equation}\label{onealg}
1=\int_{\RR} \varphi(x) \rd x=A_{n}(1)=\sum_{i=1}^n \alpha_i
\end{equation}
see \cite{hildebrand}. Moreover, note the following symmetry properties of the nodes and the weights. Let the zeros be given in increasing order, i.e., 
$x_1<\ldots<x_n$. If $n$ is even, then for $i=1,\ldots, n/2$,
\begin{align*}
x_i=-x_{n+1-i}\qquad \textnormal{ and }\qquad \alpha_i=\alpha_{n+1-i}.
\end{align*}
If $n$ is odd, then $x_{\left\lfloor n/2\right\rfloor+1}=0$ and for $i=1,\ldots,\left\lfloor n/2\right\rfloor$,
\begin{align*}
x_i=-x_{n+1-i}\qquad \textnormal{ and }\qquad \alpha_i=\alpha_{n+1-i}.
\end{align*}
We show the following estimate for the worst-case error of Gauss-Hermite rules in $\cH(K_{a,b,\omega})$. 

\begin{prop}\label{wc_ghr1}
Let $A_n$ be a Gauss-Hermite rule of order $n$. Then we have $$e^2(A_n,K_{a,b,\omega}) \leq \omega^{a(2n)^b}\frac{\sqrt{8\pi}}{1-\omega^2}.$$ 
\end{prop}

\begin{proof}
For $k\in \NN$ we have $\int_{\RR} H_k(x) \varphi(x)\rd x=0$ and hence ${\rm err}(H_k)=-\sum_{i=1}^n \alpha_i H_k(x_i)$. From these observations and from formula \eqref{formula_wc-error} we obtain 
\begin{align*}
e^2(A_{n},K_{a,b,\omega})= & \sum_{k=2n}^{\infty} \omega^{a k^b}\left(\sum_{i=1}^n \alpha_i H_k(x_i)\right)^2= \sum_{k=2n}^{\infty} 
\omega^{a k^b} {\rm err}^2(H_k).
\end{align*}
Due to the above symmetry properties it directly follows that for $l\geq n$ we have
\begin{align*}
{\rm err}(H_{2l+1})=0.
\end{align*}
For the Hermite polynomials of degree $2l$ with $l\geq n$ we proceed analogously to \cite{xiang}. Cramer's bound, see e.g., \cite[p. 324]{sansone},  states that 
\begin{align*}
\vert H_{l}(x)\vert\leq \frac{1}{\sqrt{\varphi(x)}}=\sqrt[4]{2\pi}\exp(x^2/4)\quad\textnormal{ for all } l\in\NN_0,
\end{align*}
and so we get
\begin{align*}
\left\vert{\rm err}(H_{2l})\right\vert&\leq\sum_{i=1}^{n}\alpha_i \left\vert H_{2l}(x_i)\right\vert=\sum_{i=1}^{n}\alpha_i \sqrt[4]{2\pi}\exp(x_i^2/4).
\end{align*}
Due to \cite[Section 8.7]{hildebrand} we know that
\begin{align*}
\int_{-\infty}^{\infty} \sqrt[4]{2\pi}\exp(x^2/4)\varphi(x)dx = \sum_{i=1}^{n}\alpha_i \sqrt[4]{2\pi}\exp(x_i^2/4) + 
\frac{\sqrt[4]{2\pi}}{(2n)!}\frac{\rd^{2n}}{\rd x^{2n}}\exp(x^2/4)\vert_{x=\zeta}
\end{align*}
with $\zeta\in\RR$. By induction we obtain that
\begin{align*}
\frac{\rd^{2n}}{\rd x^{2n}}\exp(x^2/4)=p_{n}(x^2)\exp(x^2/4)
\end{align*}
where $p_n$ is a polynomial of degree $n$ and with nonnegative coefficients. Consequently,
\begin{align*}
\frac{\rd^{2n}}{\rd x^{2n}}\exp(x^2/4)\vert_{x=\zeta}=p_{n}(\zeta^2)\exp(\zeta^2/4)\geq 0
\end{align*}
holds for any $\zeta\in\RR$. Thus,
\begin{align*}
\left\vert{\rm err}(H_{2l})\right\vert\leq\sum_{i=1}^{n}\alpha_i \sqrt[4]{2\pi}\exp(x_i^2/4)\leq \int_{-\infty}^{\infty} 
\sqrt[4]{2\pi}\exp(x^2/4)\varphi(x)dx=\sqrt[4]{8\pi}.
\end{align*}
This means that for $k\geq 2n$
$$
\left\vert{\rm err}(H_k)\right\vert\leq \begin{cases} \sqrt[4]{8\pi}& \textnormal{if }k\textnormal{ is even,}\\0& \textnormal{if } 
k \, \textnormal{ is odd,}\end{cases}
$$
and therefore
\begin{align*}
e^2(A_n,K_{a,b,\omega})&=\sum_{k=n}^{\infty}\omega^{a(2k)^b}{\rm err}^2(H_{2k})\leq \omega^{a(2n)^b}\frac{\sqrt{8\pi}}{1-\omega^2}.
\end{align*}
\end{proof}

\subsection{The weighted multivariate case}

For integration in the multivariate case, we use the cartesian product of one-dimensional Gauss-Hermite rules. 
Let $m_1,\ldots,m_s \in \NN$ and let $n=m_1 m_2\cdots m_s$. 
For $j=1,2,\ldots,s$ let $A_{m_j}^{(j)}(f)=\sum_{i=1}^{m_j} \alpha_i^{(j)} f(x_i^{(j)})$ be one-dimensional Gauss-Hermite rules of order 
$m_j$ with nodes $x_1^{(j)},\ldots,x_{m_j}^{(j)}$ and with weights $\alpha_1^{(j)},\ldots,\alpha_{m_j}^{(j)}$, respectively. Then we apply 
the $s$-dimensional Cartesian product rule $$A_{n,s}=A_{m_1}^{(1)} \otimes \cdots \otimes A_{m_s}^{(s)},$$ i.e.,
\begin{equation}\label{cpghr}
A_{n,s}(f)=\sum_{i_1=1}^{m_1}\ldots \sum_{i_s=1}^{m_s} \alpha_{i_1}^{(1)} 
\cdots \alpha_{i_s}^{(s)} f(\bsx_{i_1,\ldots,i_s}) \ \ \ \mbox{ for } \ f \in \mathcal{H}(K_{s,\bsa,\bsb,\omega}),
\end{equation}
where  $\bsx_{i_1,\ldots,i_s}=(x_{i_1}^{(1)},\ldots,x_{i_s}^{(s)})$. 

The following proposition provides an upper bound on the worst-case error of integration rules of the form~\eqref{cpghr}.
\begin{prop}\label{wc_ghrs}
Let $A_{n,s}$ be the $s$-dimensional Cartesian product of Gauss-Hermite rules of order $m_j$ given by~\eqref{cpghr} and let 
$n=m_1 \cdots m_s$. Then we have 
$$e^2(A_{n,s},K_{s,\bsa,\bsb,\omega}) \leq -1+\prod_{j=1}^s \left(1+\omega^{a_j (2 m_j)^{b_j}} \frac{\sqrt{8\pi}}{1-\omega^2}\right).$$ 
\end{prop}

\begin{proof}
For the worst-case error of $A_{n,s}$ in $\mathcal{H}(K_{s,\bsa,\bsb,\omega})$ we have
\begin{align*}
e^2(A_{n,s},K_{s,\bsa,\bsb,\omega})= & \sum_{\bsk \in \NN_0^s \setminus \{\bszero\}} \omega^{|\bsk|_{\bsa,\bsb}}
\left(\sum_{i_1=1}^{m_1} \ldots \sum_{i_s=1}^{m_s} \alpha_{i_1}^{(1)}\cdots \alpha_{i_s}^{(s)} H_{\bsk}(\bsx_{i_1,\ldots,i_s})\right)^2\\
= & -1 +\prod_{j=1}^s\left(1+\sum_{k=1}^\infty \omega^{a_j k^{b_j}} \left(\sum_{i=1}^{m_j} \alpha_i^{(j)} H_k(x_i^{(j)})\right)^2 \right) \\
= & -1 +\prod_{j=1}^s(1+e^2(A_{m_j},K_{a_j,b_j,\omega}))\\
\le & -1+\prod_{j=1}^s \left(1+\omega^{a_j (2 m_j)^{b_j}} \frac{\sqrt{8\pi}}{1-\omega^2}\right),
\end{align*}
where we used Proposition~\ref{wc_ghr1} for the last estimate.
\end{proof}

Based on Proposition~\ref{wc_ghrs} we now show three theorems which give sufficient conditions for UEXP, EC-SPT and EC-WT, 
respectively (Theorems~\ref{thm_suff_uexp},~\ref{thm_suff_ecspt}, and~\ref{thm_suff_ecwt}). This is achieved by a clever choice 
of the parameters $m_1,\ldots,m_s$. In the article~\cite{KPW14}, parameters $m_1,\ldots,m_s$ similar to those introduced below 
were used for numerical integration of smooth functions in Korobov spaces defined on $[0,1]^s$. In that paper, the $m_j$ defined 
a regular grid that served as integration node set. Here, we make similar choices for the parameters $m_j$, but they now determine 
the order of the Gauss-Hermite rule \eqref{cpghr}.

The first theorem in this section
shows that we can always achieve EXP and it implies the sufficient condition for UEXP in Item~\ref{intuexp} of Theorem~\ref{thmint(u)exp}.

\begin{theorem}\label{thm_suff_uexp}
For $s\in \NN$, let
$$B(s):=\sum_{j=1}^{s} \frac{1}{b_j}.$$
Furthermore, for $\e\in(0,1)$, define
$$
m=\max_{j=1,2,\dots,s}\
\left\lceil \left(
\frac{1}{a_j}\,\frac{\log\left(\frac{\sqrt{8\pi}}{1-\omega^2}\frac{s}{\log(1+\e^2)}\right)}
{\log\, \omega^{-1}}\right)^{B(s)}\,\right\rceil.
$$
Let $m_1,m_2,\ldots,m_s$
given by
$$
m_j:=\left\lfloor m^{1/(B(s) \cdot b_j)}\right\rfloor\ \ \ \ \
\mbox{for}\ \ \  j=1,2,\ldots,s\ \ \
\mbox{and}\ \ \ n=\prod_{j=1}^sm_j.
$$
Then
$$
e(A_{n,s},K_{s,\bsa,\bsb,\omega})\le\e,\ \ \ \ \mbox{and}\ \ \ \
n(\e,s)\le n=\mathcal{O}\left(\log^{\,B(s)}\left(1+\frac1\e\right)\right)
$$
with the factor in the $\mathcal{O}$ notation independent
of $\e^{-1}$ but dependent on $s$.
\end{theorem}
\begin{proof}
First note that $$n=\prod_{j=1}^s m_j=\prod_{j=1}^s \left\lfloor
m^{1/(B(s) \cdot b_j)} \right\rfloor \le m^{\frac{1}{B(s)}\sum_{j=1}^s
1/b_j} \le m=\mathcal{O}\left(\log^{\,B(s)}\left(1+\frac1\e\right)\right).
$$
Since $\lfloor x\rfloor\ge x/2$ for all $x\ge1$, we have
$$a_j(2m_j)^{b_j}\ge a_j \,m^{1/B(s)}
$$
for every $j=1,2,\ldots,s$. Then we obtain
$$
e^2(A_{n,s},K_{s,\bsa,\bsb,\omega})
 \le -1+\prod_{j=1}^s
\left(1+\omega^{a_j m^{1/B(s)}} \frac{\sqrt{8\pi}}{1-\omega^2}\right).
$$
From the definition of $m$ we have
for all $j=1,2,\dots,s$,
$$
\omega^{a_j m^{1/B(s)}} \frac{\sqrt{8\pi}}{1-\omega^2}
\le \frac{\log(1+\e^2)}{s}.
$$
This proves
$$
e(A_{n,s},K_{s,\bsa,\bsb,\omega})\le
\left[-1+\left(1+\frac{\log(1+\e^2)}{s}\right)^s\,\right]^{1/2}
\le \left[-1+\exp(\log(1+\e^2))\right]^{1/2}=\e,
$$
which completes the proof.
\end{proof}
\vskip 1pc

The following theorem shows the sufficient condition for EC-SPT in Item~\ref{intuexp} of Theorem~\ref{thmint(u)exp}.

\begin{theorem}\label{thm_suff_ecspt}
Assume that  
$$
B=\sum_{j=1}^\infty\frac1{b_j}<\infty.
$$
%and for any $p_1\in(0,1/B)$ there is a positive $\beta_1$ such that
%$$
%a_j 2^{b_j} \ge\beta_1\,2^{\,j\,p_1}\ \ \ \ \ \mbox{for all}\ \ \ \ \ j\in\NN.
%$$
Let $m_1,\ldots,m_s$ be given by
$$
m_j=\left\lceil
\left(\frac{\log\left(\frac{\sqrt{8\pi}}{1-\omega^2}\,\frac{\pi^2}{6}\,\frac{j^2}{\log(1+\e^2)}\right)}
{a_j 2^{b_j}\,\log \omega^{-1}}\right)^{1/b_j}\right\rceil.
$$
Then  $e(A_{n,s},K_{s,\bsa,\bsb,\omega}) \le\e$
and for any  positive $\delta$
there exists a positive number $C_\delta$ such that
$$
n(\e,s)\le n=\prod_{j=1}^sm_j\le C_\delta
\,\log^{\,B+\delta}\left(1+\frac1\e\right)
\ \ \ \ \mbox{for all}\ \ \ \ \e\in(0,1),\ s\in\NN.
$$
This means that we have EC-SPT with $\tau^\ast$ at most $B$.
\end{theorem}

\begin{proof}
We first prove that $e^2(A_{n,s},K_{s,\bsa,\bsb,\omega})\le\e^2$. Note that $m_j$ is defined such that
$$
\omega^{a_j(2 m_j)^{b_j}} \frac{\sqrt{8\pi}}{1-\omega^2}
\le \frac6{\pi^2}\ \frac{\log(1+\e^2)}{j^2}.
$$
Therefore
\begin{align*}
e^2(A_{n,s},K_{s,\bsa,\bsb,\omega})\le&
-1+\prod_{j=1}^s\left(1+\frac6{\pi^2}\,\frac{\log(1+\e^2)}{j^2}\right)\\
=&-1+\exp\left(\sum_{j=1}^s\log\left(1+\frac6{\pi^2}\,
\frac{\log(1+\e^2)}{j^2}\right)\right)\\
\le&-1 +\exp\left(\frac6{\pi^2}\,
\log(1+\e^2)\ \sum_{j=1}^sj^{-2}\right)\\
\le&-1+\exp\left(\log(1+\e^2)\right)=\e^2,
\end{align*}
as claimed.

We now estimate $m_j$ and then $n=\prod_{j=1}^sm_j$.
Clearly, $m_j\ge1$ for all $j\in\NN$. We prove that $m_j=1$ for large $j$.
Indeed, $m_j=1$ if
\begin{equation}\label{lasteq}
a_j 2^{b_j}\log \omega^{-1}\ge
\log\left(\frac{\sqrt{8\pi}}{1-\omega^2}\frac{\pi^2}6\,\frac{j^2}{\log(1+\e^2)}\right).
\end{equation}

Let $\delta>0$. From $\sum_k \frac{1}{b_k}< \infty$ it follows with the Cauchy condensation test that also $\sum_k \frac{2^k}{b_{2^k}}< \infty$ and hence $\lim_k 2^k/b_{2^k} =0$. Hence, we find that $b_{2^k}^{-1} \le \frac{\delta}{2^{k+1}}$ for $k$ large enough. For large enough $j$ with $2^k \le j \le 2^{k+1}$ we then obtain $$\frac{1}{b_j} \le \frac{1}{b_{2^{k}}} \le \frac{\delta}{2^{k+1}} \le \frac{\delta}{j}$$ or, equivalently, $2^{b_j} \ge 2^{j/\delta}$. Hence, there exists a positive $\beta_1$ such that $$2^{b_j} \ge \beta_1 2^{j/\delta}\ \ \ \ \mbox{ for all }\ \ \ j \in \NN.$$

Hence $a_j2^{b_j}\ge \beta_1 2^{j/\delta}$ and then the inequality \eqref{lasteq} holds for all
$j\ge j^*$, where $j^*$ is the smallest positive integer for which
$$
j^*\ge \frac{\delta}{\log\,2}\, \log\left(\frac1{\beta_1 \log \omega^{-1}}\ \log\left(
\frac{\sqrt{8\pi}}{1-\omega^2}\frac{\pi^2}6\, \frac{[j^*]^2}{\log(1+\e^2)}\right)\right).
$$
Clearly,
$$
j^*= \frac{\delta}{\log\,2}\log\ \log\ \e^{-1}+\mathcal{O}(1) \ \ \ \
\mbox{as}\ \ \ \ \e\to0.
$$
Without loss of generality we can restrict ourselves to $\e\le {\rm e}^{-{\rm e}}$, where ${\rm e}=\exp(1)$,
so that $\log\,\log\,\e^{-1}\ge1$. Then there exists a number $C_0\ge1$,
independent of $\e$ and $s$,  such that
$$
m_j=1\ \ \ \ \mbox{for all}\ \ \ \ j> \left\lfloor
C_0+\frac{\delta}{\log\,2}\,\log\ \log\ \e^{-1}\right\rfloor.
$$

We now estimate $m_j$ for $j\le \left\lfloor C_0+\frac{\delta}{\log\,2}\,
\log\,\log\,\e^{-1}\right\rfloor$. Note that
$$
\log\left(\frac{\sqrt{8\pi}}{1-\omega^2}\frac{\pi^2}6\ \frac{j^2}{\log(1+\e^2)}\right)=
\log\left(\frac{\sqrt{8\pi}}{1-\omega^2}\frac{\pi^2}{6\,\log(1+\e^2)}\right)\ +\
\log(j^2).
$$
Then $a_j 2^{b_j}\ge \beta_1 2^{j/\delta}$ also implies that
$$
C(\bsa,\bsb):= \sup_{j\in\NN}\frac{\log(j^2)}{a_j 2^{b_j}}<\infty.
$$
Furthermore, there exists a number $C_1\ge1$, independent of $\e$ and $s$
such that
$$
\log\left(\frac{\sqrt{8\pi}}{1-\omega^2}\frac{\pi^2}{6\,\log(1+\e^2)}\right)\le C_1+2\log\frac1\e\ \ \ \
\mbox{for all}\ \ \ \ \e\in(0,1).
$$
This yields
$$
m_j\le1+\left(\frac{C(\bsa,\bsb)+C_1+2\log \e^{-1}}{\log \omega^{-1}}\right)^{1/b_j}
\ \ \ \ \ \mbox{for all}\ \ \ \ j\le
\left\lfloor C_0+\frac{\delta}{\log\,2}\,\log\,\log
\frac1\e\right\rfloor.
$$
Let
$$
k=\min\left(s,\left\lfloor C_0+\frac{\delta}{\log\,2}
\,\log\,\log\frac1\e\right\rfloor\right).
$$
Then for $C=\max\left(C(\bsa,\bsb)+C_1,-2{\rm e}+\log \omega^{-1}\right)$ we have
$$
\max\left(1,\frac{C(\bsa,\bsb)+C_1+2\log\e^{-1}}{\log \omega^{-1}}\right)
\le \frac{C+2\log\e^{-1}}{\log \omega^{-1}}
$$
and
\begin{align*}
n=&\prod_{j=1}^sm_j=\prod_{j=1}^km_j\le\prod_{j=1}^k
\left(1+\left(\frac{C+2\log\e^{-1}} {\log \omega^{-1}}\right)^{1/b_j}\right)\\
=&\left(\frac{C+2\log\e^{-1}} {\log \omega^{-1}}
\right)^{\sum_{j=1}^k1/b_j}\,
\prod_{j=1}^k\left(1+
\left(\frac{\log \omega^{-1}}{C+2\log\e^{-1}}\right)^{1/b_j}\right)\\
\le&
\left(\frac{C+2\log\e^{-1}} {\log \omega^{-1}}\right)^{B}\,2^{k}.
\end{align*}
Note that
$$
2^k\le 2^{C_0}\,\exp\left(\delta \log\,\log\frac{1}{\e}\right)=2^{C_0}\,\log^{\delta}\frac1\e.
$$
Therefore there is a positive number $C_\delta$ independent of $\e^{-1}$ and $s$
such that
$$
n\le C_\delta\,\log^{B+\delta}\left(1+\frac1\e\right),
$$
as claimed. This completes the proof.
\end{proof}

Now we prove the sufficient condition for EC-WT stated in Item~\ref{wtsuff} of Theorem~\ref{thmint(u)exp}.
\begin{theorem}\label{thm_suff_ecwt}
Assume that there exist $\eta>0$ and $\beta>0$ such that $$a_j 2^{b_j} \ge \beta j^{1+\eta}\ \ \ \ \mbox{ for all }\ \ j \in \NN.$$ Then we have EC-WT.
\end{theorem}

\begin{proof}
Let $$
A:=\sum_{j=1}^\infty\frac1{a_j 2^{b_j}}<\infty
$$
and let $m_1,\ldots,m_s$ be given by
$$
m_j=\left\lceil
\left(\frac{\log\left(\frac{\sqrt{8\pi}}{1-\omega^2}\,A\,\frac{a_j 2^{b_j}}{\log(1+\e^2)}\right)}
{a_j 2^{b_j}\,\log \omega^{-1}}\right)^{1/b_j}\right\rceil.
$$
Note that $m_j$ is defined such that
$$
\omega^{a_j(2 m_j)^{b_j} } \frac{\sqrt{8 \pi}}{1-\omega^2}
\le \frac1{A}\ \frac{\log(1+\e^2)}{a_j 2^{b_j}}
$$
and therefore
\begin{align*}
e^2(A_{n,s},K_{s,\bsa,\bsb,\omega})\le&
-1+\prod_{j=1}^s\left(1+\frac1{A}\,\frac{\log(1+\e^2)}{a_j 2^{b_j}}\right)\\
=&-1+\exp\left(\sum_{j=1}^s\log\left(1+\frac1{A}\,
\frac{\log(1+\e^2)}{a_j 2^{b_j}}\right)\right)\\
\le&-1 +\exp\left(\frac1{A}\,
\log(1+\e^2)\ \sum_{j=1}^s \frac1{a_j 2^{b_j}}\right)\\
\le&-1+\exp\left(\log(1+\e^2)\right)=\e^2,
\end{align*}
as claimed.

We now estimate $m_j$ and then $\log n=\sum_{j=1}^s \log m_j$.
Clearly, $m_j\ge1$ for all $j\in\NN$. We prove that $m_j=1$ for large $j$.
Indeed, $m_j=1$ if
$$
a_j 2^{b_j}\log \omega^{-1}\ge
\log\left(\frac{\sqrt{8 \pi}}{1-\omega^2} A\,\frac{a_j 2^{b_j}}{\log(1+\e^2)}\right).
$$
This holds if and only if 
$$
\log(1+\varepsilon^2) \ge \frac{\sqrt{8 \pi}}{1-\omega^2} A a_j 2^{b_j} \omega^{a_j 2^{b_j}}. 
$$
Let $\omega_1 \in (\omega,1)$ and let $$K=\sup_{x \in \RR^+} \frac{x}{(\omega_1/\omega)^x}.$$ Then we have 
$$\frac{\sqrt{8 \pi}}{1-\omega^2} A a_j 2^{b_j} \omega^{a_j 2^{b_j}} \le \frac{K\, A \sqrt{8\pi}}{1-\omega^2} \omega_1^{a_j 2^{b_j}}.$$
Hence $$\log(1+\varepsilon^2) \ge \frac{K\, A \sqrt{8\pi}}{1-\omega^2} \omega_1^{a_j 2^{b_j}}$$ implies that $m_j=1$.
The last inequality is equivalent to $$a_j 2^{b_j} \ge \frac1{\log \omega_1^{-1}} \log \left( \frac{1-\omega^2}{K \, A \sqrt{8\pi}} \frac{1}{\log(1+\varepsilon^2)}\right).$$

Since $a_j 2^{b_j}> \beta j^{1+\eta}$ a sufficient condition for $m_j=1$ is $$\beta j^{1+\eta} \ge \frac1{\log \omega_1^{-1}} \log \left( \frac{1-\omega^2}{K \, A \sqrt{8\pi}} \frac{1}{\log(1+\varepsilon^2)}\right)= \frac{2}{\log \omega_1^{-1}} \log \frac{1}{\varepsilon}+\mathcal{O}(1).$$
Then there exists a number $C_0\ge1$,
independent of $\e$ and $s$,  such that
$$
m_j=1\ \ \ \ \mbox{for all}\ \ \ \ j> \left(\frac{2}{\beta \, \log \omega_1^{-1}} \log \frac{1}{\varepsilon}+C_0\right)^{1/(1+\eta)}.
$$

We now estimate $m_j$ for $j\le \left\lfloor \left(\frac{2}{\beta \, \log \omega_1^{-1}}\log \frac{1}{\varepsilon}+C_0\right)^{1/(1+\eta)} \right\rfloor$. Note that
$$
\log\left(\frac{\sqrt{8 \pi}}{1-\omega^2} A\ \frac{a_j 2^{b_j}}{\log(1+\e^2)}\right)=
\log\left(\frac{\sqrt{8 \pi}}{1-\omega^2} \frac{A}{\log(1+\e^2)}\right)\ +\
\log(a_j 2^{b_j}).
$$
Then $a_j 2^{b_j}\rightarrow \infty$ also implies that
$$
C(\bsa,\bsb):= \sup_{j\in\NN}\frac{\log a_j 2^{b_j}}{a_j 2^{b_j}}<\infty.
$$
Furthermore, there exists a number $C_1\ge1$, independent of $\e$ and $s$
such that
$$
\log\left(\frac{\sqrt{8 \pi}}{1-\omega^2}\frac{A}{\log(1+\e^2)}\right)\le C_1+2\log\frac1\e\ \ \ \
\mbox{for all}\ \ \ \ \e\in(0,1).
$$
This yields
$$
m_j\le1+\left(\frac{C(\bsa,\bsb)+C_1+2\log \e^{-1}}{\log \omega^{-1}}\right)^{1/b_j}
\ \ \ \ \ \mbox{for all}\ \ \ \ j\le
\left\lfloor \left(\frac{2}{\beta \, \log \omega_1^{-1}} \log \frac{1}{\varepsilon}+C_0\right)^{1/(1+\eta)} \right\rfloor.
$$
Let
$$
k=\min\left(s,\left\lfloor \left(\frac{2}{\beta\,\log \omega_1^{-1}} \log \frac{1}{\varepsilon}+C_0\right)^{1/(1+\eta)} \right\rfloor\right).
$$
Then for $C=\max\left(C(\bsa,\bsb)+C_1,\log \omega^{-1}\right)$ we have
$$
\max\left(1,\frac{C(\bsa,\bsb)+C_1+2\log\e^{-1}}{\log \omega^{-1}}\right)
\le \frac{C+2\log\e^{-1}}{\log \omega^{-1}}
$$
and
\begin{align*}
n=&\prod_{j=1}^sm_j=\prod_{j=1}^km_j\le\prod_{j=1}^k
\left(1+\left(\frac{C+2\log\e^{-1}} {\log \omega^{-1}}\right)^{1/b_j}\right)\\
=&\left(\frac{C+2\log\e^{-1}} {\log \omega^{-1}}
\right)^{\sum_{j=1}^k1/b_j}\,
\prod_{j=1}^k\left(1+
\left(\frac{\log \omega^{-1}}{C+2\log\e^{-1}}\right)^{1/b_j}\right)\\
\le&
\left(\frac{C+2\log\e^{-1}} {\log \omega^{-1}}\right)^{k}\,2^{k}.
\end{align*}
Hence $$\log n \le k \log \left(\frac{C+2\log\e^{-1}} {\log \omega^{-1}}\right) + k \log 2.$$
Note that for $\log \e^{-1} \rightarrow \infty$ we have $$k \le C_2 (\log\e^{-1})^{1/(1+\eta)}$$ with some $C_2>0$ independent of $s$ and $\varepsilon$.
Therefore we have

$$\log n \le C_2 (\log\e^{-1})^{1/(1+\eta)} \left(2+\log\left(\frac{C+2\log\e^{-1}} {\log \omega^{-1}}\right)\right)=\mathcal{O}\left((\log\e^{-1})^{1/(1+\eta)} \log \log \e^{-1}\right)$$ 
with an implied constant independent of $s$ and $\varepsilon$.

All together it follows that the logarithmic information complexity satisfies $$\log n(\varepsilon,s) = \mathcal{O}\left((\log\e^{-1})^{1/(1+\eta)} \log \log \e^{-1}\right)$$ with an implied constant independent of $s$ and $\varepsilon$.

Therefore we obtain $$\lim_{s+\log \e^{-1}\rightarrow \infty} \frac{\log n(\varepsilon,s)}{s+\log \e^{-1}} =0$$ and hence we have EC-WT as claimed. 
\end{proof}

\begin{rem}\rm
It follows easily from the above proof that the sufficient condition for EC-WT in Theorem~\ref{thmint(u)exp} and  \ref{thm_suff_ecwt} can be improved in the sense that it is enough to demand that $a_j 2^{b_j} \ge \psi(j)$ for some invertible function $\psi:\NN \rightarrow \RR^+$ satisfying $$\sum_j \frac{1}{\psi(j)} < \infty \ \ \mbox{ and }\ \  \psi^{-1}(j)=o\left(\frac{j}{\log j}\right).$$ 
\end{rem}

\begin{appendix}
\section{Analyticity of the functions in $\cH(K_{s,\bsa,\bsb,\omega})$}\label{app_a}

\begin{prop}
Let $f\in\cH(K_{s,\bsa,\bsb,\omega})$. Then $f$ is analytic.
\end{prop}
\begin{proof}
Since $\inf_j a_j\geq 1$ and $\inf_j b_j=1$, we have $\cH(K_{s,\bsa,\bsb,\omega})\subseteq\cH(K_{s,\1,\1,\omega})$ with $\1=\{1\}_{j\geq 1}$. 
Therefore, it is sufficient to show analyticity for functions $f$ which belong to $\cH(K_{s,\1,\1,\omega})$. Let $\bsl=(l_1,l_2,\ldots,l_s)\in\NN_0^s$ 
be a multiindex with $|\bsl|=l_1+l_2+\cdots+l_s$ and let $$\frac{\partial^{|\bsl|}}{\partial \bsx^{\bsl}} = \frac{\partial^{|\bsl|}}{\partial x_1^{l_1} \partial x_2^{l_2}\ldots \partial x_s^{l_s}}.$$

For any $f\in\cH(K_{s,\1,\1,\omega})$ we obtain 
\begin{align*}
\frac{\partial^{|\bsl|}}{\partial \bsx^{\bsl}}f(\bsx)=\sum_{\bsk\in\NN_0^s}\widehat{f}(\bsk)\frac{\partial^{|\bsl|}}{\partial 
\bsx^{\bsl}}H_{\bsk}(\bsx)=\sum_{\bsk\geq\bsl}\widehat{f}(\bsk)\sqrt{\frac{\bsk!}{(\bsk-\bsl)!}}H_{\bsk-\bsl}(\bsx).
\end{align*}
Then,
\begin{align*}
\left\vert \frac{\partial^{|\bsl|}}{\partial \bsx^{\bsl}}f(\bsx)\right\vert&=\left\vert \sum_{\bsk\geq\bsl}\left(\widehat{f}(\bsk)\left(\omega^{|\bsk|}\right)^{-1/2}\right)\left(\left(\omega^{|\bsk|}\right)^{1/2}\sqrt{\frac{\bsk!}{(\bsk-\bsl)!}}H_{\bsk-\bsl}(\bsx)\right)\right\vert\\
&\leq \|f\|_{\cH(K_{s,\1,\1,\omega})}\frac{1}{\sqrt{\varphi_s(\bsx)}}\left(\sum_{\bsk\geq\bsl}\omega^{|\bsk|}\frac{\bsk!}{(\bsk-\bsl)!}\right)^{1/2}\\
&=\|f\|_{\cH(K_{s,\1,\1,\omega})}\frac{1}{\sqrt{\varphi_s(\bsx)}} \left(\sum_{\bsk\geq\bsl}\prod_{j=1}^{s}(l_j!)^2\omega^{k_j}\frac{k_j!}{(k_j-l_j)!(l_j!)^2}\right)^{1/2}\\
&\leq\|f\|_{\cH(K_{s,\1,\1,\omega})}\frac{\bsl!}{\sqrt{\varphi_s(\bsx)}} \prod_{j=1}^{s}\left(\sum_{k=0}^{\infty}\binom{k}{l_j}\frac{\omega^{k}}{l_j!}\right)^{1/2}\\
&\leq\|f\|_{\cH(K_{s,\1,\1,\omega})}\frac{\bsl!}{\sqrt{\varphi_s(\bsx)}}\prod_{j=1}^{s}\left(\sum_{k=0}^{\infty}\binom{k}{l_j}\omega^k\right)^{1/2}\\
&\leq\|f\|_{\cH(K_{s,\1,\1,\omega})}\frac{\bsl!}{\sqrt{\varphi_s(\bsx)}}\prod_{j=1}^{s}\left(\frac{\omega^{l_j}}{(1-\omega)^{l_j+1}}\right)^{1/2},
\end{align*}
where $\bsl!=\prod_{j=1}^s (l_j!)$.
Now we show that $f$ can be locally represented by its Taylor expansion. For any $\bsy\in\RR^s$ and any $\bsx\in\RR^s$ with $\|\bsx-\bsy\|_\infty^2<\frac{1-\omega}{\omega}$, 
\begin{align*}
\sum_{\bsl\in\NN_0^s}\frac{1}{\bsl!}\frac{\partial^{|\bsl|}}{\partial\bsx^{\bsl}}f(\bsy)\prod_{j=1}^{s}(x_j-y_j)^{l_j}&\leq\|f\|_{\cH(K_{s,\1,\1,\omega})}\frac{1}{\sqrt{\varphi_s(\bsy)}} \sum_{\bsl\in\NN_0^s}\prod_{j=1}^{s}\left(\frac{\omega^{l_j}(x_j-y_j)^{2\ell_j}}{(1-\omega)^{l_j+1}}\right)^{1/2}\\
&\leq\|f\|_{\cH(K_{s,\1,\1,\omega})}\frac{1}{\sqrt{\varphi_s(\bsy)}} \left(\frac{1}{1-\omega}\sum_{l=0}^{\infty}\left(\frac{\omega\|\bsx-\bsy\|_\infty^{2}}{1-\omega}\right)^{l}\right)^{s/2}\\
&\leq\|f\|_{\cH(K_{s,\1,\1,\omega})}\frac{1}{\sqrt{\varphi_s(\bsy)}} \left(\frac{1}{1-\omega-\omega\|\bsx-\bsy\|_\infty^2}\right)^{s/2}<\infty.
\end{align*}
It remains to show that the remainder $R_n$ of the Taylor polynomial, given by
\begin{align*}
R_n:=\sum_{|\bsk|=n+1}\frac{n+1}{\bsk!}(\bsx-\bsy)^{\bsk}\int_{0}^{1}(1-t)^n \frac{\partial^{|\bsk|}}{\partial\bsx^{\bsk}}f(\bsy+t(\bsx-\bsy)) \rd t
\end{align*}
vanishes if $n$ goes to infinity. We have
\begin{align*}
\vert R_n\vert&\leq\sum_{|\bsk|=n+1}\frac{n+1}{\bsk!}\vert\bsx-\bsy\vert^{\bsk}\int_{0}^{1}|1-t|^n \left|\frac{\partial^{|\bsk|}}{\partial\bsx^{\bsk}}f(\bsy+t(\bsx-\bsy))\right|\rd t\\
&\leq \sum_{|\bsk|=n+1}(n+1)\vert\bsx-\bsy\vert^{\bsk}\int_{0}^{1} \|f\|_{\cH(K_{s,\1,\1,\omega})}\frac{|1-t|^n}{\sqrt{\varphi_s(\bsy+t(\bsx-\bsy))}}\prod_{j=1}^{s}\left(\frac{\omega^{k_j}}{(1-\omega)^{k_j+1}}\right)^{1/2} \rd t\\
&\leq (n+1)\|f\|_{\cH(K_{s,\1,\1,\omega})}\left[\int_{0}^{1} \frac{|1-t|^n}{\sqrt{\varphi_s(\bsy+t(\bsx-\bsy))}}\rd t \right]\sum_{|\bsk|=n+1}\prod_{j=1}^{s}\left(\frac{\omega^{k_j}|x_j-y_j|^{2k_j}}{(1-\omega)^{k_j+1}}\right)^{1/2}.
\end{align*}
Since $\|\bsx-\bsy\|_\infty<\sqrt{\frac{1-\omega}{\omega}}$, we have for any $j=1,\ldots,s$, 
\begin{align*}
\frac{1}{\sqrt{\varphi(y_j+t(x_j-y_j))}}\leq \begin{cases}\frac{1}{\sqrt{\varphi\left(y_j+\sqrt{(1-\omega)/\omega}\right)}} & \textnormal{if } y_j\geq 0\\\frac{1}{\sqrt{\varphi\left(y_j-\sqrt{(1-\omega)/\omega}\right)}} & \textnormal{if } y_j< 0 \end{cases}
\end{align*}
such that we we can bound $1/\sqrt{\varphi_s(\bsy+t(\bsx-\bsy))}$ by some constant $C_1$ independent of $n$ and $t$. Hence,
\begin{align*}
\vert R_n\vert&\leq C_1\|f\|_{\cH(K_{s,\1,\1,\omega})}(n+1)\left[\int_{0}^{1}|1-t|^n \rd t\right]\frac{1}{(1-\omega)^{s/2}}\sum_{|\bsk|=n+1}\prod_{j=1}^{s}\left(\frac{\omega|x_j-y_j|^2}{1-\omega}\right)^{k_j/2}\\
&\leq C_1\|f\|_{\cH(K_{s,\1,\1,\omega})}\frac{1}{(1-\omega)^{s/2}}\sum_{|\bsk|=n+1}\left(\frac{\omega\|\bsx-\bsy\|_\infty^2}{1-\omega}\right)^{|\bsk|/2}\\
&\leq \frac{C_1\|f\|_{\cH(K_{s,\1,\1,\omega})}}{(1-\omega)^{s/2}}\left(\frac{\omega\|\bsx-\bsy\|_\infty^2}{1-\omega}\right)^\frac{n+1}{2}\frac{(s+n)!}{(s-1)!(n+1)!}.
\end{align*}
Since $\omega\|\bsx-\bsy\|_\infty^2/(1-\omega)<1$ and $(s+n)!/((s-1)!(n+1)!)=O(n^{s-1})$, we get that $R_n\rightarrow 0$ as $n$ goes to $\infty$. Thus, $f$ is analytic.
\end{proof}

\section{An example}\label{app_b}

Let $f:\RR^s\rightarrow \RR$ be given as $f(\bsx)=f(x_1,\ldots,x_s)=\exp\left(\frac{1}{\sqrt{s}}\sum_{j=1}^{s}x_j\right)$. 
We now show that $f\in\cH(K_{s,\bsa,\bsone,\omega})$ for any weight sequences $\bsa$. 
The exponential generating function of the Hermite polynomials is given by $$\exp\left(t x-\frac{t^2}{2}\right)=\sum_{l=0}^{\infty}\frac{t^{l}}{\sqrt{l!}} H_{l}(x),$$
see \cite[p.\ 7]{B98}, and thus, we get for $t=\frac{1}{\sqrt{s}}$ that
\begin{align*}
\exp\left(\frac{x}{\sqrt{s}}\right)={\rm e}^{\frac{1}{2s}}\sum_{l=0}^{\infty}\frac{t^{l}}{\sqrt{l!}} H_{l}(x).
\end{align*}
For any $\bsk\in\NN_0^s$ the $\bsk$th Hermite coefficient of $f$ is
\begin{align*}
\widehat{f}(\bsk)&=\int_{\RR^s}f(\bsx)H_{\bsk}(\bsx)\varphi_s(\bsx)\rd \bsx\\
&=\prod_{j=1}^{s}\int_{\RR}\exp\left(\frac{x_j}{\sqrt{s}}\right)H_{k_j}(x_j)\varphi(x_j)\rd x_j\\
&=\sqrt{{\rm e}}\prod_{j=1}^{s}\sum_{l=0}^\infty\frac{1}{\sqrt{l!s^l}}\int_{\RR}H_{l}(x_j)H_{k_j}(x_j)\varphi(x_j)\rd x_j\\
&=\sqrt{{\rm e}}\prod_{j=1}^{s}\frac{1}{\sqrt{k_j!s^{k_j}}}.
\end{align*}
Hence, 
\begin{align*}
\|f\|_{K_{s,\bsa,\bsone,\omega}}&={\rm e} \sum_{\bsk\in\NN_0^s} \prod_{j=1}^{s}\omega^{-a_j{k_j}^{b_j}} \frac{1}{k_j!s^{k_j}}= \exp\left(1+\frac{1}{s}\sum_{j=1}^s \omega^{-a_j}\right)< \infty,
\end{align*}
and therefore $f \in \cH(K_{s,\bsa,\bsone,\omega})$ as desired.

\end{appendix}

\vspace{1cm}
\noindent{\bf Authors' Address:}\\

\noindent Institut f\"{u}r Finanzmathematik, Johannes Kepler Universit\"{a}t Linz, Altenbergerstra{\ss}e 69, A-4040 Linz, Austria.\\
\noindent Email: \{christian.irrgeher,peter.kritzer,gunther.leobacher,friedrich.pillichshammer\}@jku.at

\end{document}